\documentclass{article}
\usepackage{amssymb}
\usepackage{amsfonts}
\usepackage{amsmath}
\usepackage{hyperref}
\usepackage{graphicx}
\hypersetup{urlcolor=blue,linkcolor=blue ,citecolor=blue,colorlinks=true}
\newtheorem{theorem}{Theorem}[section]

\newtheorem{definition}[theorem]{Definition}
\newtheorem{example}[theorem]{Example}

\newtheorem{lemma}[theorem]{Lemma}

\newtheorem{remark}[theorem]{Remark}

\newenvironment{proof}[1][Proof]{\textbf{#1.} }{\ \rule{0.5em}{0.5em}}

\begin{document}

\title{Some results about Mittag-Leffler function's integral representations}
\author{  Yayun Wu\\
{\small  \textit{Beijing Normal University, Beijing 100875,
China} }}
\maketitle

\begin{abstract}
In many articles on the integral expressions of Mittag-Leffler functions, we have found that whether the integral expression can be used at the origin is still unresolved. In this article we give the applicable conditions and proof. And we also generalize some important conclusions about Mittag-Leffler function.
\end{abstract}
\section{Introduction}
In recent years, the Mittag-Leffler functions and Mittag-Leffler type functions have caused more and more interests among scientists, engineers and applications. This interest comes from the close connection of these functions to fractional differential equations.
The Mittag-Leffler function
\begin{equation*}
E_{\alpha}(z)=\sum_{k=0}^{\infty}\frac{z^{k}}{\Gamma(k\alpha+1)},\alpha>0,z\in\mathbb{C}.\eqno(1)
\end{equation*}
was firstly studied by Mittag-Leffler. An important generalization,
\begin{equation*}
E_{\alpha,\beta}(z)=\sum_{k=0}^{\infty}\frac{z^{k}}{\Gamma(k\alpha+\beta)},\alpha>0,\beta\in\mathbb{C},z\in\mathbb{C},\eqno(2)
\end{equation*}
was introduced in$[6,7]$. These two functions can be viewed as the result of generalization of the exponential function,
\begin{equation*}
exp(z):=\sum_{n=0}^{\infty}\frac{z^{n}}{n!}=\sum_{n=0}^{\infty}\frac{z^{n}}{\Gamma(n+1)}.\eqno(3)
\end{equation*}
More investigations of properties of the Mittag-Leffler function and its applications to fractional differential equations and related questions have been carried out by $[1-4,8,9]$.

Integral representations play a prominent role in the analysis of the Mittag-Leffler function. In $[10]$, the authors study integral represents of $E_{\alpha,\beta}(z)$ in the open right half-plane when $\alpha\in(0,2),\beta\in(0,1+\alpha)$ by using the well known $Hankel$ formula about gamma function.The authors study the distribution of zeros of $E_{\alpha,\beta}(z)$ relying on the integral represent. In $[11]$,the authors also use $Hankel's$ integral contour acquire the integral representation of $E_{\alpha,\beta}(z),\alpha>0,\beta\in\mathbb{R},z\in\mathbb{C}$. They use the results to study the algorithms for numerical evalution of the Mittag-Leffler function.

Another method to study integral represent of Mittag-Leffler function is Lapalce inversion.In $[12]$,
Gorenflo and Mainardi use the Laplace inversion integral to give the integral represent of  $E_{\alpha}(-t^{\alpha}),\alpha\in(0,3)$ and use these results to search fractional oscillations. For more results about Mittag-Leffler's integral represent and its applications,one can see $[5,13]$.By using $Hankel$ contour or Laplace inversion method, the original will be the branch point. We want to know that whether the integral represent can be applied at original.The previous articles have not explained this problem. In this paper, we will give the proof.

In this paper,we use the Laplace inversion method to acquire the integral represent of $E_{\alpha,\beta}(\lambda t^{\alpha}),\alpha\in(0,1),\beta\in(0,1+\alpha),\lambda\in\mathbb{C},t>0,$ and study its asymptotic properties. When $\beta=1$ or $\beta=\alpha$,we can obtain the integral represent of $E_{\alpha}(\lambda t^{\alpha})$ and $E_{\alpha,\alpha}(\lambda t^{\alpha})$.
In fact,$E_{\alpha}(\lambda t^{\alpha})$ and $E_{\alpha,\alpha}(\lambda t^{\alpha})$ are often used in fractional differential equations and application problems.

\section{Preliminaries}

In this section,we give some useful definitions and related results. One can see$[5,8,9,16]$ for more details.
\begin{definition}
\label{DE2.3} The $Mittag-Leffler$ function is defined by the following
formula
\begin{equation*}
E_{\alpha,\beta}(z):=\sum_{k=0}^{\infty}\frac{z^{k}}{\Gamma(k\alpha+\beta)}%
,E_{\alpha}(z):=E_{\alpha,1}(z),\eqno(4)
\end{equation*}
where $z\in \mathbb{C},\alpha>0,\beta>0$.$\Gamma(\cdot)$ is gamma function.
\end{definition}

\begin{definition}
\label{DE2.4} Let $f(t)$ be an arbitrary function defined on the interval $0< t<\infty$; then
\begin{equation*}
\mathcal{L}\{f(t)\}(s)=\int_{0}^{\infty}e^{-st}f(t)dt,\eqno(5)
\end{equation*}
is the Laplace transform, provided that the integral exists. And if $f(t)$ is of exponential order $e^{at}$, then the Laplace transform of $f(t)$ exists for all provided $\Re s>a$.
\end{definition}
\begin{example}$([17])$Let $\alpha>0,\beta>0,\Re(s)>0,|\lambda|<|s^{\alpha}|$, then we have
\begin{equation*}
\mathcal{L}\{t^{\beta-1}E_{\alpha,\beta}(\lambda t^{\alpha})\}(s)=\frac{s^{\alpha-\beta}}{z^{\alpha}-\lambda}.\eqno(6)
\end{equation*}
\end{example}
\begin{lemma}
\label{DE2.5}{\bf (Waston's Lemma)}Suppose that the $f(t)$ has the asymptotic expansion:
\begin{equation*}
f(t)\sim\sum_{v=1}^{\infty}a_{v}t^{\lambda_{v}},t\rightarrow0+,\eqno(7)
\end{equation*}

\begin{equation*}
-1<Re(\lambda_{1})<Re(\lambda_{2})<Re(\lambda_{3})<\cdots;
\end{equation*}
then $F(p)$ has the corresponding asymptotic expansion
\begin{equation*}
F(p)\sim\sum_{v=1}^{\infty}\frac{a_{v}\Gamma(\lambda_{v}+1)}{p^{\lambda_{v}+1}},\arg(p)\in(-\frac{\pi}{2},\frac{\pi}{2}),|p|\rightarrow\infty.\eqno(8)
\end{equation*}
\end{lemma}

\begin{theorem}{\bf (The Bromwich Inversion Theorem)} Let $f(t)$ have a continuous derivative and let $|f(t)|<Ke^{\gamma t}$
where $K$ and $\gamma$ are positive constants.Define
\begin{equation*}
\mathcal{F}(s)=\mathcal{L}[f](s)=\int_{0}^{\infty}e^{-st}f(t)dt,\Re s>\gamma.
\end{equation*}
Then
\begin{equation*}
f(t)=\frac{1}{2\pi i}\lim_{T\rightarrow\infty}\int_{c-iT}^{c+iT}e^{st}\mathcal{F}(s)ds, c>\gamma.\eqno(9)
\end{equation*}
 \end{theorem}
\begin{definition}
The Mellin transform of a function $ f:\mathbb{R}_{+}\rightarrow \mathbb{C}$ is the function $f^{*}$ defined by
\begin{equation*}
f^{*}(s)=\mathcal{M}[f](s)=\int_{0}^{\infty}x^{s-1}f(x)dx,
\end{equation*}
where $a<\Re(s)<b,a,b\in\mathbb{R}.$ Here $a,b$ provided that integral exists.
\end{definition}
\begin{example}$([16,p.58])$ Let $\Re(s)>-1,|\varphi|<\frac{\pi}{2}$, then
\begin{equation*}
\mathcal{M}[e^{-x\cos(\varphi)}\sin(x\sin(\varphi))](s)=\Gamma(s)\sin(\varphi s).\eqno(10)
\end{equation*}
\end{example}
\begin{example}$([16,p.60])$
Let $\mathcal{M}[f](s)=f^{*}(s)$, $\mathcal{M}[g](s)=g^{*}(s)$,and $\Re(\lambda)>0$, $\alpha>0,$ then by definition $2.6$ and Fubini theorem, we can have
\begin{equation*}
\mathcal{M}[\int_{0}^{\infty}f(x^{\alpha}u)g(\lambda u)du](s)=\frac{\lambda^{\frac{s}{\alpha}-1}}{\alpha}g^{*}(1-\frac{s}{\alpha})f^{*}(\frac{s}{\alpha}).
\eqno(11)\end{equation*}
\end{example}

\section{Main results}

In this part,we use the Laplace inversion method to study the integral represent of $E_{\alpha,\beta}(\lambda t^{\alpha}),\alpha\in(0,1),\beta\in(0,1+\alpha)$.In $[12]$,the author discuss the integral represent of $E_{\alpha}(-t^{\alpha})$. In that paper,the author didn't give the specific calculation process.In order to explain the problem in detail, we believe it is necessary to give the specific calculation process.Our results will be more general and useful for studying fractional differential equations.
\begin{theorem}Let $t\in(0,\infty)$ $\alpha\in(0,1)$,$0<\beta<\alpha+1$, $\lambda\in\mathbb{C}.$ $\Gamma:=\overset{\frown}{AB}\cup\overline{BC}\cup\overset{\frown}{CD}\cup\overline{DE}\cup \overset{\frown}{EF}\cup\overline{FA}$ denotes the integral contour and the direction leaves the region on the left(Figure $1$).We cut along the negative axis,and the origin is the branch point. $\overline{BC}$ and $\overline{DE}$ represent the upper and low land, respectively. $\overset{\frown}{CD}$ represents the arc around the origin. Choosing $c>\max\{\Re\lambda,0\}$,and
$$\begin{cases}
\overset{\frown}{AB}:z=c-Re^{i\theta},\theta\in( -\frac{\pi}{2},0),\\
\overline{BC}:z=r,r\in(c-R,-\rho),\\
\overset{\frown}{CD}:z=\rho e^{-i\theta},\theta \in(-\pi,\pi),\\
\overline{DE}:z=-r,r\in(\rho,R-c),\\
\overset{\frown}{EF}:z=c-Re^{i\theta},\theta\in(0,\frac{\pi}{2}),\\
\overline{FA}:z=c+iu,u\in(-R,R).
\end{cases}
$$

\begin{figure}
  \centering
  \includegraphics[height=7cm,width=6cm,angle=0]{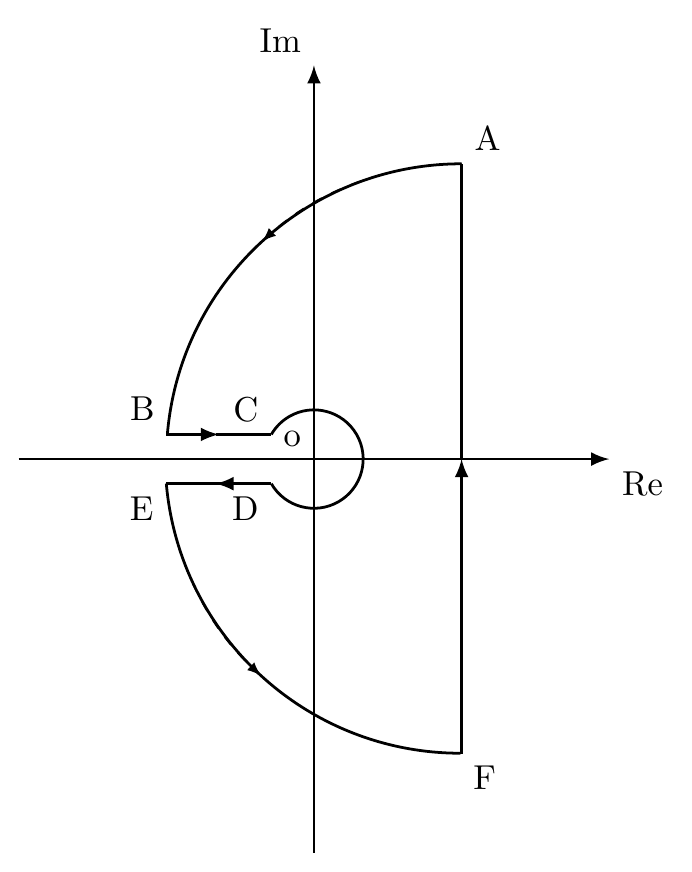}
  \caption{The modified Bromwhich contour.}\label{fig:f1}
\end{figure}

then
\begin{equation*}
\frac{1}{2\pi i}\int_{\Gamma}e^{zt}\frac{z^{\alpha-\beta}}{z^{\alpha}-\lambda}dz
=\mathop{Res}\limits_{z=\lambda^{\frac{1}{\alpha}}}\frac{e^{zt}z^{\alpha-\beta}}{z^{\alpha}-\lambda}.\eqno(12)
\end{equation*}
And let $\rho\rightarrow 0,$ $R\rightarrow \infty$,we have
\begin{equation*}
t^{\beta-1}E_{\alpha,\beta}(\lambda t^{\alpha})=\mathop{Res}\limits_{z=\lambda^{\frac{1}{\alpha}}}\frac{e^{zt}z^{\alpha-\beta}}{z^{\alpha}-\lambda}+\frac{1}{\pi}\int_{0}^{\infty}e^{-vt}v^{\alpha-\beta}\frac{v^{\alpha}\sin(\beta\pi)+\lambda\sin(\alpha-\beta)\pi}{v^{2\alpha}-2\lambda v^{\alpha}\cos(\alpha\pi)+\lambda^{2}}dv.\eqno(13)
\end{equation*}
\begin{proof}
$(1)$ By using Residue theorem,one can have
\begin{equation*}
\frac{1}{2\pi i}\int_{\Gamma}e^{zt}\frac{z^{\alpha-\beta}}{z^{\alpha}-\lambda}dz=\mathop{Res}\limits_{z=\lambda^{\frac{1}{\alpha}}}\frac{e^{zt}z^{\alpha-\beta}}{z^{\alpha}-\lambda}.
\end{equation*}
(2)We take the following steps to prove the other result.

(I): When $z\in\overset{\frown}{AB}$,then $\left|\frac{z^{\alpha-\beta}}{z^{\alpha}-\lambda}\right|\leq \frac{(2R)^{\alpha-\beta}}{(|R|-c)^{\alpha}-|\lambda|}$.And $\frac{2\theta}{\pi}<\sin\theta<\theta$,$\theta\in (0,\frac{\pi}{2})$.
\begin{equation*}
\left|\frac{1}{2\pi i}\int_{\overset{\frown}{AB}}e^{zt}\frac{z^{\alpha-\beta}}{z^{\alpha}-\lambda}dz
\right|\leq \frac{1}{2\pi}\left|\int_{-\frac{\pi}{2}}^{0}e^{(c-R\cos\theta)t}\frac{2^{\alpha-\beta}R^{\alpha-\beta+1}}{(R-c)^{\alpha}-|\lambda|}d\theta\right|
\end{equation*}

\begin{equation*}
\leq \frac{2^{\alpha-\beta}R^{\alpha-\beta+1}e^{ct}}{2\pi[(R-c)^{\alpha}-|\lambda|]}\int_{0}^{\frac{\pi}{2}}e^{-\frac{2R}{\pi}\theta}d\theta
\rightarrow 0, (R\rightarrow \infty).
\end{equation*}

$(II):$ When $z\in \overset{\frown}{EF}$,by the similar method in $(I)$,we can prove
\begin{equation*}
\frac{1}{2\pi i}\int_{\overset{\frown}{EF}}e^{zt}\frac{z^{\alpha-\beta}}{z^{\alpha}-\lambda}dz\rightarrow 0,(R\rightarrow \infty).
\end{equation*}

$(III):$ When $z\in \overset{\frown}{CD}$,we choose small $\rho$,so that $|z^{\alpha}-\lambda|>\frac{|\lambda|}{2}>0$,
\begin{equation*}
\left|\frac{1}{2\pi i}\int_{\overset{\frown}{CD}}e^{zt}\frac{z^{\alpha-\beta}}{z^{\alpha-\lambda}}dz\right|
\leq \frac{1}{2\pi}\left|\int_{-\pi}^{\pi}e^{\rho e^{i\theta}}\frac{(\rho e^{i\theta})^{\alpha-\beta}}{(\rho e^{i\theta})^{\alpha}-\lambda}\rho e^{i \theta}id\theta\right|
\end{equation*}

\begin{equation*}
\leq
\frac{\rho^{\alpha-\beta+1}}{\pi|\lambda|}\int_{-\pi}^{\pi}e^{\rho\cos\theta}d\theta
\leq \frac{2\rho^{\alpha-\beta+1 }e^{\rho}}{|\lambda|}\rightarrow 0,(\rho\rightarrow 0).
\end{equation*}

$(IV):$  When $z\in H(\rho):= \overline{BC}\cup\overset{\frown}{CD}\cup\overline{DE}$,and let $\rho$ tends to $0$, $R$ tends to $\infty$,
\begin{equation*}
\lim_{\mathop{\rho\rightarrow 0}\limits_{R\rightarrow\infty}}\frac{1}{2\pi i}\int_{H(\rho)} e^{zt}\frac{e^{(\alpha-\beta)Ln(z)}}{e^{\alpha Ln(z)}-\lambda}dz
\end{equation*}

\begin{equation*}
=\frac{1}{2\pi i}\left(\int_{-\infty}^{0}e^{zt}
\frac{e^{(\alpha-\beta)(ln|z|+i\pi)}}{e^{\alpha(ln|z|+i\pi)}-\lambda}dz
+\int_{0}^{-\infty}e^{zt}\frac{e^{(\alpha-\beta)(ln|z|-i\pi)}}{e^{\alpha(ln|z|-i\pi)}-\lambda}dz
\right)\end{equation*}

\begin{equation*}
=\frac{1}{2\pi i}\left(\int_{0}^{\infty}e^{-vt}\frac{v^{\alpha-\beta}e^{i(\alpha-\beta)\pi}}{v^{\alpha}e^{i\alpha\pi}-\lambda}dv -\int_{0}^{\infty}e^{-vt}\frac{v^{\alpha-\beta}e^{-i(\alpha-\beta)\pi}}{v^{\alpha}e^{-i\alpha\pi}-\lambda}dv\right)
\end{equation*}

\begin{equation*}
=-\frac{1}{\pi}\int_{0}^{\infty}e^{-vt}v^{\alpha-\beta}\frac{v^{\alpha}\sin(\beta\pi)+\lambda\sin(\alpha-\beta)\pi}{v^{2\alpha}-2\lambda v^{\alpha}\cos(\alpha\pi)+\lambda^{2}}dv
\end{equation*}

$(V)$: When $z\in \overline{FA}$,and let $R$ tends to $\infty$,using the Laplace inversion formula
\begin{equation*}
\lim_{R\rightarrow\infty}\frac{1}{2\pi i}\int_{c-iR}^{c+iR}e^{zt}\frac{z^{\alpha-\beta}}{z^{\alpha}-\lambda }dz=t^{\beta-1}E_{\alpha,\beta}(\lambda t^{\alpha}).
\end{equation*}
Summarizing the above results,the proof have been completed.
\end{proof}
\end{theorem}
\begin{remark}In theorem $3.1$,we have
\begin{equation*}
\mathop{Res}\limits_{z=\lambda^{\frac{1}{\alpha}}}\frac{e^{zt}z^{\alpha-\beta}}{z^{\alpha}-\lambda}
=\begin{cases}
0,&\lambda<0,\\
\frac{1}{\alpha}e^{\lambda^{\frac{1}{\alpha}}t}\lambda^{\frac{1-\beta}{\alpha}},&\arg(\lambda^{\frac{1}{\alpha}})\in(-\pi,\pi).
\end{cases}\eqno(14)
\end{equation*}
\end{remark}
\begin{remark}
From the theorem $3.1$, when $\alpha\in(0,1),\lambda\nleq0,\arg(\lambda)\in(-\pi\alpha,\pi\alpha)$,
$t>0$,then
\begin{equation*}
E_{\alpha}(\lambda t^{\alpha})=\frac{1}{\alpha}e^{t\lambda^{\frac{1}{\alpha}%
}}-\frac{\lambda \sin \left( \pi\alpha \right) }{\pi}\int_{0}^{\infty}%
\frac{e^{-ut}u^{\alpha-1}}{u^{2\alpha}-2\lambda u^{\alpha} \cos \left(\pi\alpha
\right)+\lambda^{2}}du,\eqno(15)
\end{equation*}
\begin{equation*}
t^{\alpha-1}E_{\alpha,\alpha}(\lambda t^{\alpha})=\frac{\lambda^{\frac{1}{\alpha}-1}}{\alpha}e^{t\lambda^{\frac{1}{\alpha}%
}}+\frac{\sin(\pi\alpha) }{\pi}\int_{0}^{\infty}%
\frac{e^{-ut}u^{\alpha}}{u^{2\alpha}-2\lambda u^{\alpha} \cos \left(\pi\alpha
\right)+\lambda^{2}}du.\eqno(16)
\end{equation*}
\end{remark}

In theorem $3.1$,the original point ia a branch point.By the following theorem,we will prove that the integral represent$(14),(15)$ can also be true when $t=0$.Our conclusion complement the theorem $1'$ in $[10]$.

\begin{theorem}
 Let $\alpha\in(0,1),\beta\in(0,1+\alpha),\arg(\lambda)\in(-\pi\alpha,\pi\alpha),$ we denote \begin{equation*}f_{\alpha,\beta}(v)=\frac{1}{\pi}v^{\alpha-\beta}\frac{v^{\alpha}\sin(\beta\pi)+\lambda\sin(\alpha-\beta)\pi}{v^{2\alpha}-2\lambda v^{\alpha}\cos(\alpha\pi)+\lambda^{2}},\eqno(17)
\end{equation*}
and \begin{equation*}M_{1}=\int_{0}^{\infty}\frac{v^{2\alpha-\beta}}{v^{2\alpha}-2\lambda v^{\alpha}\cos(\alpha\pi)+\lambda^{2}}dv,\eqno(18)
\end{equation*}
\begin{equation*}
M_{2}=\int_{0}^{\infty}\frac{v^{\alpha-\beta}}{v^{2\alpha}-2\lambda v^{\alpha}\cos(\alpha\pi)+\lambda^{2}}dv,\eqno(19)
\end{equation*}
then we have
\begin{equation*}
\int_{0}^{\infty}f_{\alpha,\beta}(v)dv=\frac{\sin(\beta\pi)}{\pi}M_{1}
+\frac{\lambda\sin\pi(\alpha-\beta)}{\pi}M_{2}
\end{equation*}
\begin{equation*}
=\begin{cases}
-\frac{\lambda^{\frac{1-\beta}{\alpha}}}{\alpha},&\beta\in(1,\alpha+1),\\
1-\frac{1}{\alpha},&\beta=1,\\
nonexist,&\beta=\alpha.
\end{cases}
\end{equation*}
\end{theorem}
\begin{proof} $(I)$If $\beta\in(1,\alpha+1)$,we have
\begin{equation*}
M_{1}=\frac{1}{\alpha}\int_{0}^{\infty}\frac{v^{\alpha-\beta+1}}{(v^{\alpha}-\lambda e^{i\pi\alpha})(v^{\alpha}-\lambda e^{-i\pi\alpha})}dv^{\alpha}
\end{equation*}

\begin{equation*}
=\frac{1}{\alpha}\int_{-\infty}^{\infty}\frac{e^{(2-\frac{\beta-1}{\alpha})x}}{(e^{x}-\lambda e^{i\pi\alpha})(e^{x}-\lambda e^{-i\pi\alpha})}dx.
\end{equation*}
In order to obtain the result of the above integral,we need to calculate the following integral:
\begin{equation*}
\frac{1}{\alpha}\int_{\Gamma}\frac{e^{(2-\frac{\beta-1}{\alpha})z}}{(e^{z}-\lambda e^{i\pi\alpha})(e^{z}-\lambda e^{-i\pi\alpha})}dz
\end{equation*}
in which $\Gamma$ is the integral contour,$\Gamma=\overline{AB}+\overline{BC}+\overline{CD}+\overline{DA}( Figure 2)$,
$$\begin{cases}
\overline{AB}:z=x,x\in(-R,R)\\
\overline{BC}:z=R+iy,y\in(0,2\pi)\\
\overline{CD}:z=-x+i2\pi,x\in(-R,R)\\
\overline{DA}:z=-R-iy,y\in(-2\pi,0)
\end{cases}$$

\begin{figure}
  \centering
  \includegraphics[height=7cm,width=9cm,angle=0]{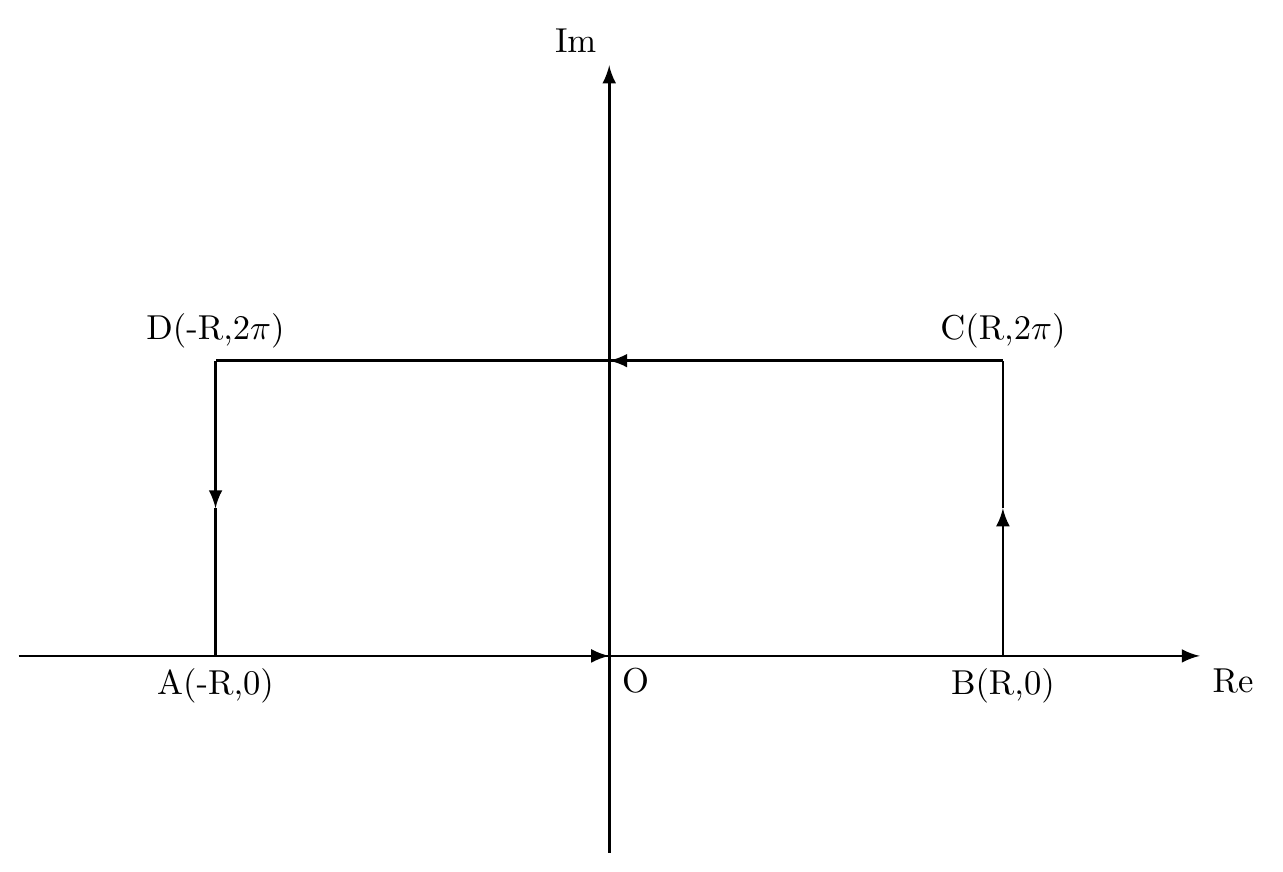}
  \caption{Integral contour.}\label{fig:f1}
\end{figure}

\begin{equation*}
\int_{\Gamma}=\int_{\overline{AB}}+\int_{\overline{BC}}+\int_{\overline{CD}}+\int_{\overline{DA}}.
\end{equation*}
By using residue theorem and let $R\rightarrow \infty$,then we have
\begin{equation*}
M_{1}=\frac{2\pi i}{\alpha[1-e^{\frac{(1-\beta)2\pi i}{\alpha}}]}\left(\frac{e^{(1+\frac{1-\beta}{\alpha})z_{1}}}{e^{z_{1}}-\lambda e^{-i\pi\alpha}}+\frac{e^{(1+\frac{1-\beta}{\alpha})z_{2}}}{e^{z_{2}}-\lambda e^{i\pi\alpha}}\right),
\end{equation*}
\begin{equation*}
=\frac{\pi\lambda^{\frac{1-\beta}{\alpha}}\sin(\pi(\beta-\alpha+\frac{1-\beta}{\alpha}))}{\alpha\sin(\pi\alpha)\sin(\pi\frac{\beta-1}{\alpha})},
\end{equation*}
in which $e^{z_{1}}=\lambda e^{i\pi\alpha},e^{z_{2}}=\lambda e^{i\pi(2-\alpha)}$. Using the same method,we have
\begin{equation*}
M_{2}=\frac{\pi\lambda^{\frac{1-\beta}{\alpha}-1}\sin((\beta+\frac{1-\beta}{\alpha})\pi)}{\alpha\sin(\pi\alpha)\sin((\frac{\beta-1}{\alpha})\pi)},
\end{equation*}
and we note the fact that
\begin{equation*}
\sin(\pi(\beta-\alpha+\frac{1-\beta}{\alpha}))\sin(\beta\pi)-\sin(\pi(\beta-\alpha))\sin(\pi(\beta+\frac{1-\beta}{\alpha}))
\end{equation*}
\begin{equation*}
=\sin(\pi\frac{1-\beta}{\alpha})\sin(\pi\alpha),\eqno(20)
\end{equation*}
So we can have
\begin{equation*}
\int_{0}^{\infty}f_{\alpha,\beta}(v)dv=-\frac{\lambda^{\frac{1-\beta}{\alpha}}}{\alpha}.
\end{equation*}

$(II)$ If $\beta=1$,
\begin{equation*}
\int_{0}^{\infty}f_{\alpha,1}(v)dv=
-\frac{\lambda\sin(\pi\alpha)}{\pi}M_{2},
\end{equation*}
in which
$$\begin{array}{ll}
M_{2}&=\int_{0}^{\infty}\frac{v^{\alpha-1}}{v^{2\alpha}-2\lambda v^{\alpha}\cos(\alpha\pi)+\lambda^{2}}dv\\
&=\frac{1}{\alpha}\int_{0}^{\infty}\frac{du}{(u-\lambda e^{i\pi\alpha})(u-\lambda e^{-i\pi\alpha})}\\
&=\frac{\pi(1-\alpha)}{\alpha\lambda\sin(\alpha\pi)}.
\end{array}$$
So we have
\begin{equation*}
\int_{0}^{\infty}f_{\alpha,1}(v)dv=1-\frac{1}{\alpha}.
\end{equation*}

$(III)$ If $\beta=\alpha$,we can easily find $v^{\alpha}|v^{2\alpha}-2\lambda v^{\alpha}cos(\pi\alpha)+\lambda^{2}|^{-1}\simeq O(v^{-\alpha})$ when $v\rightarrow{\infty}$, so $\int_{0}^{\infty}f_{\alpha,1}(v)dv$ doesn't exist.
\end{proof}
\begin{remark}When $\alpha\in(\frac{1}{2},1),\beta\in(1,1+\alpha),\Re(\lambda)>0,$ by using Example $2.7$ and $2.8$, we can have the same results.In fact, if $\Re(s)>-1,$
$$\begin{array}{ll}
F(s)&=\mathcal{M}[\int_{0}^{\infty}e^{-(\lambda-u^{\alpha}\cos(\pi\alpha))v}\sin(u^{\alpha}v\sin(\pi\alpha))dv](s)\\
&=\mathcal{M}[\int_{0}^{\infty}e^{u^{\alpha}v\cos(\pi(1-\alpha))}\sin(u^{\alpha}v\sin(\pi(1-\alpha)))e^{-\lambda v}dv](s)\\
&=\frac{\lambda^{\frac{s}{\alpha}-1}}{\alpha}\Gamma(1-\frac{s}{\alpha})\Gamma(\frac{s}{\alpha})\sin(\pi(1-\alpha)\frac{s}{\alpha})\\
&=\frac{\pi\lambda^{\frac{s}{\alpha}-1}}{\alpha\sin(\frac{s\pi}{\alpha})}\sin(\pi(1-\alpha)\frac{s}{\alpha}).
\end{array}$$
And
\begin{equation*}
M_{1}=\frac{F(\alpha-\beta+1)}{\sin(\pi\alpha)}=\frac{\pi\lambda^{\frac{1-\beta}{\alpha}}\sin(\pi(\beta-\alpha+\frac{1-\beta}{\alpha}))}{\alpha\sin(\pi\alpha)\sin(\pi\frac{\beta-1}{\alpha})},
\end{equation*}
\begin{equation*}
M_{2}=\frac{F(1-\beta)}{\sin(\pi\alpha)}=\frac{\pi\lambda^{\frac{1-\beta}{\alpha}-1}\sin((\beta+\frac{1-\beta}{\alpha})\pi)}{\alpha\sin(\pi\alpha)\sin((\frac{\beta-1}{\alpha})\pi)}.
\end{equation*}
\end{remark}

\section{Applications}
We use  $(14),(15)$ to generalize the lemma $2$ in $[14]$ and give a new proof.
\begin{lemma}
Let $\alpha\in(0,1),\arg(\lambda)\in(-\pi\alpha,\pi\alpha),$
\begin{equation*}
J_{\lambda}(t)=\int_{0}^{\infty}e^{-ut}\frac{u^{\alpha-1}}{u^{2\alpha}-2%
\lambda u^{\alpha} \cos \left( \pi\alpha \right) +\lambda^{2}}du,\eqno(21)
\end{equation*}
then $J_{\lambda}(t)$ is continuous on $[0,\infty)$ and
\begin{equation*}
\lim\limits_{t\rightarrow\infty}J_{\lambda}(t)=0.
\end{equation*}
\end{lemma}

\begin{proof}In fact,when $\arg(\lambda)\in(-\pi\alpha,\pi\alpha)$,for $\forall u\in(0,\infty)$,$u^{2\alpha}-2\lambda u^{\alpha}\cos(\pi\alpha)+\lambda^{2}=(u^{\alpha}-\lambda e^{-\pi\alpha})(u^{\alpha}-\lambda e^{\pi\alpha})\neq0$. Choosing small $\varepsilon>0,$when $u\in (0,\varepsilon)$, we can have
$|u^{2\alpha}-2\lambda u^{\alpha}cos\pi\alpha+\lambda^{2}|>\frac{
|\lambda|^{2}}{2}>0$.When  $u\rightarrow\infty,$
\begin{equation*}
\frac{u^{\alpha-1}}{|u^{2\alpha}-2\lambda u^{\alpha}cos\pi\alpha+\lambda^{2}|%
}=\mathcal{O}(u^{-\alpha-1}).
\end{equation*}
And we can have
\begin{equation*}
|J_{\lambda}(t)|\leq \frac{2\Gamma(\alpha)}{|\lambda|^{2}t^{\alpha}}%
+\int_{\varepsilon}^{\infty}\frac{e^{-ut}u^{\alpha-1}}{|u^{2\alpha}-2\lambda
u^{\alpha}cos\pi\alpha+\lambda^{2}|}du,
\end{equation*}
so $\lim\limits_{t\rightarrow\infty}|J_{\lambda}(t)|=0.$
And with theorem $3.4$,we have $J_{\lambda}(t)$ is continuous on $[0,\infty)$ and
$\lim\limits_{t\rightarrow\infty}J_{\lambda}(t)=0.$
\end{proof}
\begin{theorem} Let $\alpha\in(0,1)$,$\lambda\neq0$, $\arg(\lambda)\in(-\pi\alpha,\pi\alpha)$, then
\begin{equation*}
E_{\alpha}(\lambda t^{\alpha})-\frac{1}{\alpha}e^{\lambda^{\frac{1}{\alpha}t}}\sim\sum_{k=1}^{\infty}\frac{a_{k}\Gamma(k\alpha)}{t^{k\alpha}},t\rightarrow\infty,
\eqno(22)\end{equation*}
\begin{equation*}
t^{\alpha-1}E_{\alpha,\alpha}(\lambda t^{\alpha})-\frac{\lambda^{\frac{1-\alpha}{\alpha}}}{\alpha}e^{\lambda^{\frac{1}{\alpha}}t}\sim\sum_{k=1}^{\infty}\frac{b_{k}\Gamma(k\alpha+1)}{t^{k\alpha+1}},t\rightarrow\infty,
\eqno(23)\end{equation*}
in which $
a_{k}=-\frac{ e^{i\pi\alpha}\sin(\pi\alpha)}{\pi\lambda^{k}}(\frac{\sin(k+1)\pi}{\sin(\pi\alpha)}-e^{ik\pi\alpha}),
$
$b_{k}=-\frac{1}{\lambda} a_{k}$,$k=1,2,\cdots.$ And there will have $t_{0}>0$, so that
\begin{equation*}
|E_{\alpha}(\lambda t^{\alpha})-\frac{1}{\alpha}e^{\lambda^{\frac{1}{\alpha}}t}|\leq \frac{K_{\alpha,\lambda}}{t^{\alpha}},t\in(0,\infty),
\eqno(24)\end{equation*}
\begin{equation*}
|t^{\alpha-1}E_{\alpha,\alpha}(\lambda t^{\alpha})-\frac{\lambda^{\frac{1-\alpha}{\alpha}}}{\alpha}e^{\lambda^{\frac{1}{\alpha}}t}|\leq\frac{L_{\alpha,\lambda}}{t^{\alpha+1}},[t_{0},\infty).
\eqno(25)\end{equation*}
in which $K_{\alpha,\lambda}$,$L_{\alpha,\lambda}$ are constants with $\alpha,\lambda$.

\begin{proof}In fact, when $v\in(0,|\lambda|^{\frac{1}{\alpha}})$, from $(17)$ in theorem $3.4$, we have
$$\begin{array}{ll}
f_{\alpha,1}(v)&=-\frac{\lambda\sin(\pi\alpha)}{\pi}\frac{v^{\alpha-1}}{v^{2\alpha}-2\lambda v^{\alpha}\cos(\pi\alpha)+\lambda^{2}}\\
&=-\frac{\lambda e^{i\pi\alpha}\sin(\pi\alpha)}{\pi}\sum_{k=1}^{\infty}\frac{v^{k\alpha-1}}{\lambda^{k+1}}\left(\frac{\sin(k+1)\pi}{\sin(\pi\alpha)}-e^{ik\pi\alpha}\right)\\
&=\sum_{k=1}^{\infty}a_{k}v^{k\alpha-1}
\end{array}$$
 and
\begin{equation*}
f_{\alpha,\alpha}(v)=-\frac{vf_{\alpha,1}}{\lambda}
\end{equation*}
by using lemma $2.4$, one can easily have $(22),(23)$.

Now let's prove $(24)$ and $(25)$.
In fact, we denote $f_{n}(v)=f_{\alpha,1}(v)-\sum\limits_{k=1}^{n}a_{k}v^{k\alpha-1}$,
when $v\in(0,|\lambda|^{\frac{1}{\alpha}})$,there will exist $N_{0}\geq1$,$C_{N_{0}}>0$, so that $|f_{N_{0}}(v)|<C_{N_{0}}v^{N_{0}\alpha}$. We choose proper $p\in\mathbb{C}$,$\Re(p)>0$ and denote $\varphi_{N_{0}}(v)=\int_{|\lambda|^{\frac{1}{\alpha}}}^{v}e^{-pu}f_{N_{0}}(u)du$.From lemma $4.1$, we can see that there will exist $A_{N_{0}}>0$, so that $|\varphi_{N_{0}}(v)|\leq A_{N_{0}}.$
\begin{equation*}
|E_{\alpha}(\lambda t^{\alpha})-\frac{1}{\alpha}e^{\lambda^{\frac{1}{\alpha}}t}|=|\mathcal{L}[f_{\alpha,1}](t)|\leq\sum_{k=1}^{N_{0}}\frac{|a_{k}\Gamma(k\alpha)|}{t^{k\alpha}}+|\mathcal{L}[f_{N_{0}}](t)|.
\end{equation*}
When $t>\Re(p)$, we have
\begin{equation*}
|\mathcal{L}[f_{N_{0}}](t)|\leq|\int_{0}^{|\lambda|^{\frac{1}{\alpha}}}e^{-vt}f_{N_{0}}(v)dv|+|\int_{|\lambda|^{\frac{1}{\alpha}}}^{\infty}e^{-vt}f_{N_{0}}(v)dv|
\end{equation*}
\begin{equation*}
\leq \frac{C_{N_{0}}\Gamma(N_{0}\alpha+1)}{t^{N_{0}\alpha+1}}+\frac{A_{N_{0}}|t-p|e^{-(t-\Re(p))|\lambda|^{\frac{1}{\alpha}}}}{t-\Re(p)}.
\end{equation*}
So there has $T_{0}>0,$ when $t>T_{0}$,
\begin{equation*}
|E_{\alpha}(\lambda t^{\alpha})-\frac{1}{\alpha}e^{\lambda^{\frac{1}{\alpha}}t}|\leq \frac{|a_{1}\Gamma(\alpha)|}{t^{\alpha}},
\end{equation*}
and given that $J_{\lambda}(t)$ is continuous in $[0,\infty)$,so we can have $L_{0}>0$,
$|E_{\alpha}(\lambda t^{\alpha})-\frac{1}{\alpha}e^{\lambda^{\frac{1}{\alpha}}t}|\leq\frac{L_{0}}{t^{\alpha}} ,$ $t\in(0,T_{0}].$ We choose $K_{\alpha,\lambda}\geq\max\{|a_{1}\Gamma(\alpha)|,L_{0}\}$,so we can prove $(24)$.We can use the same way to prove $(25)$.
\end{proof}
\end{theorem}

\end{document}